\documentclass[12pt]{article}
\usepackage[]{fontenc}
\usepackage[latin1]{inputenc}

\usepackage{hyperref}

\usepackage{comment}
\usepackage{amsfonts}
\usepackage{color}
\usepackage{graphicx}
\usepackage{url}
\usepackage{stmaryrd}
\usepackage{enumitem}
\usepackage{mathtools}
\usepackage{amsthm,amsmath,amssymb,color} % define this before the line numbering.
\usepackage{authblk}
\newcommand{\mrm}[1]{\mathrm{#1}}

\newcommand{\Ent}{\mathrm{Ent}}

\DeclareMathOperator*{\argmin}{arg\,min}

\DeclareMathOperator{\var}{var}
\DeclareMathOperator{\Var}{\mathbb{V}ar}
\newcommand{\bigO}{\mathcal{O}}

\newcommand{\E}{\mathbb{E}}

\newcommand{\R}{\mathbb{R}}

\newcommand{\Z}{\mathbb{Z}}
\renewcommand{\P}{\mathbb{P}}
\renewcommand{\L}{\mathbb{L}}

\newcommand{\calA}{\mathcal{A}}
\newcommand{\calC}{\mathcal{C}}

\newcommand{\calP}{\mathcal{P}}

\newcommand{\eps}{\varepsilon}
\newcommand{\ph}{\varphi}

\newcommand{\Proba}{\mathcal{P}}

\newcommand{\e}{\mathrm{e}}
\renewcommand{\d}{\mathrm{d}}

\newcommand{\eqdef}{ \dps \mathop{=}^{{\rm def}} }
\newcommand{\one}{ \, {\rm l} \hspace{-.7 mm} {\rm l}}
\newcommand{\dps}{\displaystyle}

\newcommand{\abs}[1]{\left | #1\right |}
\newcommand{\set}[1]{\left\{#1\right\}}
\newcommand{\p}[1]{ \left(#1\right) }
\renewcommand{\b}[1]{\left [ #1\right ]}

\theoremstyle{plain}
\newtheorem{The}{Theorem}[section]
\newtheorem{Lem}[The]{Lemma}
\newtheorem{Pro}[The]{Proposition}

\newtheorem{Cor}[The]{Corollary}

\newtheorem*{Ass*}{Assumption}

\newtheorem{Rem}[The]{Remark}
\newtheorem{Exa}[The]{Example}

\numberwithin{equation}{section}

\theoremstyle{definition}

\title{Entropy minimizing distributions are worst-case optimal importance proposals}
\date{}
\begin{document}

\author[1]{Fr\'ed\'eric C\'erou}
\author[1]{Patrick H\'eas}
\author[1]{Mathias Rousset}

%Mathias Rousset$^1$}

%\author[1]{Wanderer}
%\author[1]{Static}
\affil[1]{IRMAR and Inria, University of Rennes, France.}
%\affil[2]{Both on a bus}

\maketitle
\abstract{Importance sampling of target probability distributions belonging to a given convex class is considered. Motivated by previous results, the cost of importance sampling is quantified using the relative entropy of the target with respect to proposal distributions. Using a reference measure as a reference for cost, we prove under some general conditions that the worst-case optimal proposal is precisely given by the distribution minimizing entropy with respect to the reference within the considered convex class of distributions. The latter conditions are in particular satisfied when the convex class is defined using a push-forward map defining atomless conditional measures. Applications in which the optimal proposal is Gibbsian and can be practically sampled using Monte Carlo methods are discussed.}

\tableofcontents

\section{Introduction}
\emph{Importance Sampling} (IS) is a generic Monte Carlo methodology that aims at computing averages with respect to a given probability distribution $\eta \in \Proba(E)$ in a state space $E$, usually called \emph{the target distribution}, by using weighted samples distributed according to a different distribution $\mu\in \Proba(E)$, usually called the \emph{proposal distribution}.\medskip 

A a general concept, IS is at the basis of most Monte Carlo strategies ever since its introduction in computational statistical physics in the early fifties. The non-expert reader may consult the second chapter of the monograph~\cite{liu2001} or the review paper~\cite{tokdar2010importance} as an introduction. One well-known problem is the lack of robustness in the choice of the proposal distribution which leads to the degeneracy of the importance weights, especially in high dimension (see \textit{e.g.} \cite{au2003important}). Motivated in part by this issue, a considerable amount of sophisticated strategies incorporating IS have been developed. Broadly speaking, two type of ideas have emerged in all trends. First, one can smooth the sampling task by considering a pre-defined family or flow of targets and then perform sampling sequentially by starting with the easiest ones. Second, on may try to optimize the choice of the proposal, using some information on the model, or previous sampling attempts. These ideas are developed both in Sequential Monte Carlo (a.k.a. Particle Filters) methods in which the proposal itself is known up to a normalization and sampled iteratively (classical papers include~\cite{neal2001annealed,del2006sequential}), or in methods relying on a proposals in a parametric family with some adaptive features (mainly by optimization of the proposal based on empirical estimation of the 'cross-entropy' of the target \cite{rubinstein2004cross,douc2007convergence}).   \medskip
%Let $\eta$ denote a given target distribution according to which one may wish to sample using some kind of Monte Carlo procedure.  \medskip

All the above mentioned methods do suggest the following question: what is the optimal proposal associated to a given target distribution, or to a given subset of target distributions ? \medskip

In many applications, the target $\eta$ is given through a non-normalized density with respect to a reference distribution $\pi \in \Proba(E)$; $\pi$ being very easy to sample. The present work is motivated by the following additional context: the evaluation of the density defining $\eta$ is computationally very expensive, but the user knows \emph{a priori} that $\eta$  belongs to a known \emph{convex} class of 'admissible' distributions denoted $\calC \subset \Proba(E)$, which is described in a much more simple way. For instance, $\calC$ may be rigorously given by bounds on the considered model. Another scenario may occur when preliminary attempts of sampling the target $\eta$ have been performed, yielding a confidence set $\calC$ to which the true target belongs with a very high probability. \medskip

We will address and give a quite generic answer to the following problem: which importance proposal $\mu$ is (worst-case) optimal when the set of admissible target distributions is $\calC$ ? The results will be stated for a subset of admissible target distributions $\calA \subset \calC$ with appropriate properties, but we will restrict -- without significant loss of generality -- in the present introduction to the case $\calA = \calC$ for expository purpose .\medskip

The first choice we have to make is the choice of the \emph{cost} of performing importance sampling. A way to do that is through the \emph{required sample size} $N^\ast$ of i.i.d. $\mu$-distributed variables. It can be defined as the sample size required to yield at least a $\delta > 0$ accuracy with probability $1-p_\alpha$ when estimating properly normalized test functions; $\delta,p_\alpha$ being given. Instead of using the classical Chebyshev lower bound to estimate $N^\ast$ using variance, we will rather use the relative entropy (Kullback-Leibler divergence):
\begin{equation}\label{eq:cost}
\ln N^\ast \simeq \Ent(\eta \mid \mu).
\end{equation}
%%%%%\textcolor{red}{plutot $\mu$ que $\pi$ ?}
As will be discussed in Section~\ref{sec:sample_size}, some \emph{lower and upper} bounds estimates are given in~\cite{ChaDia} which show that, under some uniform tail conditions on the density $\eta/\mu$, \eqref{eq:cost} is a rigorous equivalent when $\Ent(\eta \mid \mu)$ becomes large. This is not the case of variance, as will be demonstrated in Section~\ref{sec:example}. Some care is needed however with this argument: the estimates are not sharp and quite inaccurate for non-asymptotic practical purpose. In this paper, we will reformulate and comment the results of~\cite{ChaDia} in Section~\ref{sec:sample_size}, and then simply proceed with our analysis using relative entropy as a \emph{definition} of the logarithmic cost of importance sampling. Note that relative entropy is also extensively used as a cost function in the so-called \emph{cross-entropy or adaptive importance sampling methods} (\cite{rubinstein2004cross,douc2007convergence}) in order to optimize importance proposal distributions, see Section~\ref{sec:cross-ent}. Recent works have also questioned the use of variance in practical empirical estimations (the so-called Effective Sample Size) of the divergence between the target and the proposal, proposing to consider other R\'enyi entropies (see~\cite{elvira2022rethinking,martino2017effective,huggins2019sequential}, and Section~\ref{app:ChaDia} for comments on R\'enyi entropies).

Now that we have a cost functional to compare importance sampling between various proposals, one must deal with the problem that worst-case costs are usually infinite. We simply solve this issue by performing a comparison of the log-cost of importance sampling between: i) a proposal $\mu$, and ii) the reference proposal $\pi$. If $\calC \subset \Proba(E)$ is a given convex set of admissible target distributions, one is led to define the \emph{worst-case logarithmic cost} -- of importance sampling with $\mu$ as compared to importance sampling with $\pi$ -- by the quantity:
\begin{equation}\label{eq:log-cost_intro}
 \mrm{WLC}_h(\mu \mid \pi) \eqdef  \sup_{\eta \in \calC, \, \Ent(\eta \mid \pi) \leq h} \Ent( \eta \mid \mu) - h .
\end{equation}
The quantity~\eqref{eq:log-cost_intro} is the logarithm of the (normalized) worst-case cost for importance sampling with $\mu$, for targets with a given maximal reference log-cost.\medskip %(as defined by importance sampling with $\pi$).\medskip

We will first prove in this paper (Theorem~\ref{th:main1}), that, under some specific assumptions, the worst-case optimal proposal distribution is unique and is given by the classical entropy minimizing distribution associated with $\calC$:% that contains $\calA \subset \calC$ ($\calA = \calC$ in many cases):
\begin{equation}\label{eq:result}
\mu_\ast = \argmin_{\mu \in \Proba(E)}  \mrm{WLC}_h(\mu \mid \pi) =  \argmin_{\mu \in \calC} \Ent( \mu \mid \pi).
\end{equation}
We will show more precisely that the optimal worst-case log-cost is given by $$\mrm{WLC}_h(\mu_\ast \mid \pi) + h  = h - \Ent(\mu_\ast \mid \pi) \quad \geq 0,$$
and that the difference with an other proposal is given by
$$
\mrm{WLC}_h(\mu \mid \pi) - \mrm{WLC}_h(\mu_\ast \mid \pi) \geq \Ent(\mu \mid \mu_\ast).
$$
The main sufficient condition ensuring these results is the following: any strict half-space (defined by measurable functions) containing the optimizer $\mu_\ast \in \calC$ must also contain a distribution $\eta \in \calC$ with prescribed relative entropy $\Ent(\eta \mid \pi)=h$ (at least up to an arbitrary precision).\medskip 

This abstract condition will be made more concrete in Theorem~\ref{th:main2}, where a sufficient setting is proposed in the case where the set admissible targets is defined through a push-forward map $T$:
$$\calC = \set{\eta: \, T \sharp \eta \in \calC_T \subset \Proba(F)},$$
$T \sharp \eta$ denoting the push-forward by any measurable map $T: E \to F $ and $\calC_T$ denoting an arbitrary convex set. A sufficient condition $(H_T)$ ensuring the abstract condition $(H)$ and the main results stated above is then: i) the conditional distribution $\pi( \, \, .  \mid T =t )$ has an atomless distribution $T \sharp \pi (\d t)$-almost everywhere, and ii) $\calC$ contains distributions defined as indicator densities with respect to $\mu_\ast$. A counterexample on a two atom discrete space is provided showing the necessity of the atomless assumption. \medskip

An important more practical example is for $T$ vector valued and $\calC=\set{\eta: \, \eta(T) \in C}$ for $C$ convex. It will be discussed in Section~\ref{sec:gibbs}. In that case the optimal proposal belongs to the Gibbs (canonical) exponential family:
$$
\mu_\ast \propto \exp( \langle \beta_\ast , T \rangle) \d \pi;
$$
and can be simulated using some Monte Carlo procedure (\textit{e.g.} Sequential Monte Carlo, or Direct), albeit in a cheapest way as compared to the target $\eta$.\medskip

Finally, it is interesting to remark that our main min-max characterization theorem is similar, albeit different, from the classical Pythagorean theorem for relative entropy in information geometry. We will recall in Section~\ref{sec:main} that the latter is equivalent to the min-max property:
\[
 \mu_\ast = \argmin_{\mu \in \Proba(E)} \sup_{\eta \in \calC}   \Ent(\eta \mid \mu ) - \Ent(\eta \mid \pi),
\]
which holds in general, without specific assumption. It will be discussed in Section~\ref{sec:intro_min} why the latter is not very satisfactory for a practical interpretation in terms of importance sampling. The main problem is that the relative log-cost $\Ent(\eta \mid \mu ) - \Ent(\eta \mid \pi)$ quantifies the \emph{relative improvement} of importance sampling by $\mu$ as compared $\pi$, and it turns out that for the optimal proposal $\mu=\mu_\ast$, the worst-case improvement is always attained for the 'cheapest' trivial target $\eta =\mu_\ast$; yet, in practice, one is not interested in improving the cheapest targets of the admissible set $\calC$. Our main results might be interpreted as a variant of this min-max formulation of the Pythagorean theorem in which the supremum  has to be reached by target distributions with large relative entropy $\Ent(\eta \mid \pi)  \gg 1$, which are relevant in importance sampling.\medskip 
% propertyadapted to the discussed context of importance sampling. Our main point is that the loss in genericity required by this strengthening is mild.

The paper is structured as follows. We will recall in Section~\ref{sec:intro_min} the definition and the main properties ({\textit{e.g.} the Pythagorean theorem) of the distribution $\mu_\ast$ that minimizes entropy relative to a reference $\pi$ over a convex subset. In Section~\ref{sec:main}, we will then state and prove the two main theorems of this work, Theorem~\ref{th:main1} and Theorem~\ref{th:main2}, with a counter-example for discrete state spaces. Some comments on application to Gibbs exponential families will be presented in Section~\ref{sec:gibbs}. Finally, a reformulation of, and some comments on, the results of the reference~\cite{ChaDia} on the sample size required for importance sampling will be done in Section~\ref{sec:sample_size}.

\section*{Notation} $T \sharp \pi$ denotes the push-forward (measure image) of $\pi$ by the map $T$. $\pi(\ph) = \int \ph \d \pi$ denotes integration of test functions. $ \eta / \pi $ denotes the Radon-Nikodym derivative between two non-negative measures with domination relation $\eta \ll \pi$.

\section{Entropy minimizing distributions}\label{sec:intro_min}
In this section, we recall some basic facts about the entropy minimizing distribution $\mu_\ast$ associated with a given convex subset $\calC$ of probability distributions and a reference probability $\pi$. \medskip

\subsection{Definition and Pythagorean theorem}
Let $(E,\pi)$ denotes the pair given by a (standard Borel) state space $E$ endowed with reference probability distribution $\pi$. Let $\calC \subset \Proba(E)$ be a convex subset. For simplicity, we assume there exists an \emph{entropy minimizing distribution} (with finite entropy) associated with the pair $(\pi,\calC)$. It will be  denoted
\begin{equation}\label{eq:mini}
 \mu_\ast \eqdef \argmin_{\mu \in \calC} \Ent(\mu \mid \pi).
\end{equation}
By strict convexity of entropy, the latter is uniquely defined. A general notion of entropy minimizing distribution exists in information geometry,  where it is called \emph{information projection}~\cite{csiszar2003information}. \medskip

The `Euler-Lagrange' equation or `first-order condition' characterizing~\eqref{eq:mini} are known in information geometry as the \emph{Pythagorean theorem} for relative entropy. 
\begin{The}[Theorem~$1$,~\cite{csiszar2003information}]\label{th:entmin} Let $(E,\pi)$ be a probability space. Let $\calC \subset \Proba(E)$ be convex and contain the entropy minimizing distribution denoted $\mu_\ast$. The following condition on  $\mu \in \Proba(E)$:
\begin{equation}\label{eq:Pyth}
 \forall \eta \in \calC, \qquad \Ent(\eta \mid \pi) \geq \Ent(\eta \mid \mu) + \Ent(\mu_\ast \mid \pi),%\min_{\mu\textcolor{red}{'} \in \calC} \Ent(\mu\textcolor{red}{'} \mid \pi).
\end{equation}
has a unique solution given by $\mu_\ast$.
%\textcolor{red}{ou pour etre en coherence avec le theoreme 3.1, on peut peut etre ecrire
%$$
%\forall \mu \in \calC, \qquad \Ent(\mu \mid \pi) \geq \Ent(\mu \mid \mu_*) + \Ent(\mu_* \mid \pi)?
%$$}
\end{The}
\begin{Rem} $\mu_\ast$ is also the unique distribution satisfying the more constraining condition:
\begin{equation*}
 \forall \eta \in \calC, \qquad \Ent(\eta \mid \pi) \geq \Ent(\eta \mid \mu) + \Ent(\mu \mid \pi).%\min_{\mu\textcolor{red}{'} \in \calC} \Ent(\mu\textcolor{red}{'} \mid \pi).
\end{equation*}
\end{Rem}
\begin{Rem}
The above theorem is a kind of first-order optimality condition. To see this, let us denote by $D_\mu$ the formal (Fr\'echet) derivative on $\Proba(E)$ for the usual addition of measures. At least in a formal sense, $D_\mu \Ent( \mu \mid \pi)$ is a test function, and  if $\eta \in \calC$, the difference of probability measures $\eta - \mu$ is a tangent vector pointing inside $\calC$. Formal first-order optimality condition precisely requires that:
\begin{align*}
0 &\leq  (\eta - \mu)\p{ D_\mu \Ent( \mu \mid \pi) } = (\eta-\mu)( \ln \mu/\pi) \\ 
& = \Ent(\eta \mid \pi) -  \Ent(\eta \mid \mu) - \Ent(\mu \mid \pi), \\
%& \leq   \Ent(\eta \mid \textcolor{red}{\pi}) -  \Ent(\eta \mid \mu) - \inf_{\mu' \in \calC} \Ent(\mu' \mid \pi).
\end{align*}
hence the above condition.
\end{Rem}
The following direct corollary will be useful in our proofs.
\begin{Cor} Under the assumptions of Theorem~\ref{th:entmin}, if $\eta \in \calC$ satisfies $\eta\p{\abs{\ln \frac{\mu_\ast}{\pi}} }< + \infty$, then it holds
\begin{equation*}
 \eta(\ln \frac{\mu_\ast}{\pi} ) \geq \mu_\ast(\ln \frac{\mu_\ast}{\pi}). %, \quad \text{and} \quad \eta((\ln \frac{\mu_\ast}{\pi})_-)< + \infty. %\tag{EL($\mu$)} 
\end{equation*}
\end{Cor}

\subsection{Min-max formulation}
It is especially of interest to the present work to reformulate the above theorem as a min-max problem as follows:
\begin{Cor}\label{cor:optim}
Under the assumptions of Theorem~\ref{th:entmin}, one has
 $$
 \mu_\ast = \argmin_{\mu \in \Proba(E)} \sup_{\eta \in \calC} \Ent(\eta \mid \mu) - \Ent(\eta \mid \pi).
 $$
 Moreover, for $\mu = \mu_\ast$, the supremum is attained by $\eta= \mu_\ast$:
 $$
 \sup_{\eta \in \calC} \Ent(\eta \mid \mu_\ast) - \Ent(\eta \mid \pi) = - \Ent(\mu_\ast \mid \pi).
 $$
\end{Cor}
\begin{proof} The condition~\eqref{eq:Pyth} is equivalent to the condition
$$\sup_{\eta \in \calC} \Ent(\eta \mid \mu) - \Ent(\eta \mid \pi) \leq - \Ent(\mu_\ast \mid \pi), $$
so that the statement of the Pythagorean theorem is equivalent to the two conditions: i) $\sup_{\eta \in \calC} \Ent(\eta \mid \mu) - \Ent(\eta \mid \pi) > - \Ent(\mu_\ast \mid \pi) $ if $\mu \neq \mu_\ast$, and ii) $
 \sup_{\eta \in \calC} \Ent(\eta \mid \mu_\ast) - \Ent(\eta \mid \pi) \leq - \Ent(\mu_\ast \mid \pi) $ in which the upperbound is attained for $\eta = \mu_\ast$.
\end{proof}

 %Let us define relative discount factor importance sampling of the target $\nu$, but the    attained for  Indeed, the quantity that is maximized is the ratio of the costs of importance sampling by $\mu$ and $\pi$ respectively; and the supremum is attained (perhaps not uniquely) by $\mu_\ast$ which is not a challenging target distribution. One is indeed not interested in improving the cost of the easiest target. This will briefly discuss in Section~\textcolor{red}{\ref{sec:intro_min}}. 
 
Admitting that we quantify the logarithm of the cost of importance sampling of a target $\eta\in \mathcal{C}$ with a proposal $\mu$ using relative entropy, Corollary~\ref{cor:optim} is already of interest to our sampling interpretation. Indeed let us define the \emph{worst relative logarithmic cost} by
$$\mrm{WRLC}_{\calC}(\mu \mid \pi) = \sup_{\eta \in \calC} \Ent(\eta \mid \mu) - \Ent(\eta \mid \pi)$$, that is the logarithm of worst-case \emph{improvement ratio} between performing importance sampling with proposal $\mu$ as compared to with reference proposal $\pi$. The Pythagorean theorem is equivalent to say that the entropy minimizing distribution $\mu_\ast$ is the unique optimal proposal in terms of $\mrm{WRLC}_{\calC}(\mu \mid \pi)$. The latter is nonetheless not fully satisfactory from a practical perspective because for $\mu=\mu_\ast$ the worst case target is attained by the proposal itself $\eta=\mu_\ast$. Unfortunately, it might happen in principle that $\eta =\mu_\ast$ is the unique worst-case target in the sense of the above $\mrm{WRLC}$ criteria. The Pythagorean theorem thus says nothing \textit{a priori} about the improvement of cost for those targets $\eta$ with higher relative entropy which are of practical significance.\medskip 

In a way, it is the purpose of the present paper to modify the min-max formulation of Corollary~\ref{cor:optim} in order to consider \emph{absolute rather than relative} {worst-case} costs. This justifies the definition of the worst-case log-cost~\eqref{eq:log-cost_intro}. \medskip

\subsection{Gibbs case}
We end this section with the most standard example of entropy minimizing distributions.

\begin{Exa}[Gibbs exponential family] Let $T: E \to \R^d$ be given with finite exponential moments $ \pi(\e^{\langle \beta, T\rangle }) < +\infty$ for all $\beta \in \R^d$. Assume that $$ \calC = \set{\eta \mid \eta(T) \in C} $$ where $C$ is closed convex with $C \cap \mrm{supp}(T \sharp \pi) \neq \emptyset$. Then,
$$
\mu_{\ast} = \mu_{\beta^\ast}= \frac{1}{Z_{\beta_\ast}}{\rm e}^{\langle \beta_\ast, T\rangle } \d \pi,
$$
for a unique $\beta_\ast \in \R^d$. Morever, $\beta_\ast$ is the unique $\beta \in \R^d$ satisfying the first order optimality condition
$$
\forall t \in C, \qquad \langle \beta, t \rangle \geq \langle \beta, \mu_\beta(T) \rangle .
$$
In the case where $C =\set{t_0}$, $\calC$ is called a \emph{linear family}, the equality case is satisfied in the Pythagorean theorem, and $\beta_\ast$ is the unique $\beta$ satisfying
$$
\mu_\beta(T) = t_0.
$$
\end{Exa}

\section{Main results}\label{sec:main}

We now assume in this section that one wants to compare the cost of importance sampling between: i) a proposal $\mu$, and ii) the reference proposal $\pi$. \medskip

We will use the exponential of relative entropy (rather than variance) to quantify the sample size $N^\ast$ required by importance sampling of  $\eta$ by $\mu$, that is
$
\Ent(\eta \mid \mu) \simeq \ln N^\ast .
$
This will be thoroughly discussed and justified in Section~\ref{sec:sample_size}. From now on, $\Ent(\eta \mid \mu)$ will be called the \emph{log-cost} of performing importance sampling of target $\eta$ with proposal $\mu$.  \medskip

Denoting by $\calA \subset \Proba(E)$ a given set of admissible target distributions, it is then natural to define the \emph{worst-case log-cost} -- of importance sampling with $\mu$ as compared to importance sampling with $\pi$ -- by the quantity~\eqref{eq:log-cost_intro}, that we recall here:
\begin{equation}\label{eq:log-cost}
 \mrm{WLC}_h(\mu \mid \pi) \eqdef \sup_{\eta \in \calA, \, \Ent(\eta \mid \pi) \leq h} \Ent( \eta \mid \mu) - h .
\end{equation}
The quantity~\eqref{eq:log-cost} is the logarithm of the worst-case cost of importance sampling with proposal $\mu$ for targets with a given maximal reference log-cost of importance sampling with $\pi$.\medskip

Note that since $\mu_\ast$ is entropy minimizing, the definition of $\mrm{WLC}_h$ is well-defined only for $h \geq \Ent(\mu_\ast \mid \pi)$. For $h_\ast = \Ent(\mu_\ast \mid \pi)$, the set defining the supremum is given by the singleton $\mu_\ast$, and one trivially obtains 
 $\mrm{WLC}_h(\mu \mid \pi) = \Ent( \mu_\ast \mid \mu ) - \Ent( \mu_\ast \mid \pi ) $ which has minimum $0$ for $\mu = \mu_\ast$. \medskip

In~\eqref{eq:log-cost}, the 'worst-case scenario' is defined using the subset $\set{\Ent(\,\, . \mid \pi) \leq h} \cap \calA$ defined by the admissible target distributions with a given maximal reference log-cost of importance sampling (with the reference proposal $\pi$). As detailed in the previous section, the definition of log-cost~\eqref{eq:log-cost} has to be compared with the variant of the Pythagorean theorem in Corollary~\eqref{cor:optim}. The latter involves the worst-case \emph{relative} log-cost, where the relative log-cost is defined by the difference $\Ent( \eta \mid \mu) - \Ent( \eta \mid \pi)$, whereas in~\ref{eq:log-cost} the comparison is made with a reference worst-case value $h$. This ensures that the minimization of the worst-case log-cost $\mrm{WLC}_h$ really focuses on the worse target distributions.

\subsection{The general min-max theorem}
We can then state the first of the two main theorems of this paper. The main issue is to give a general condition under which the optimal proposal distribution minimizing the worst-case log-cost $\mrm{WLC}_h$ is indeed the entropy minimizing over a convex set $\calC$ of distributions containing $\calA$:
$$
\calA \subset \calC.
$$

In short, this condition asks that in any half-space (as defined by bounded measurable functions) containing strictly $\mu_\ast$, one can find a target $\eta \in \calA$ with log-cost $\Ent(\eta \mid \pi)$ arbitrarily close to the reference log-cost $h$; or, in other words, that $\mu_\ast$ and those target distributions $\eta$ with $\Ent(\eta \mid \pi)$ close to $h$ cannot be strictly separated by an hyperplane. \medskip

We will in fact show a much more precise results. First we will show that the optimal worst-case log-cost obtained for $\mu = \mu_\ast$ is given in fact by the opposite of the minimal entropy on $\calC$: $-\Ent(\mu_\ast \mid \pi)$. This is similar to what happens in the Pythagorean theorem of Theorem~\ref{th:entmin}. We will also prove the inequality~\eqref{eq:cond_main} below, that states that the difference between i) a worst-case log-cost for any proposal $\mu$, and ii) the optimal worst-case log-cost for proposal $\mu_\ast$, is in fact greater than the log-cost of $\mu$ itself, $\Ent(\mu \mid \mu_\ast)$. This will immediately yield the desired characterization of $\mu_\ast$ as optimal worst-case proposal, with a quantification of optimality given by $\Ent(\mu \mid \mu_\ast)$.

\begin{The}\label{th:main1}Let $(E,\pi)$ be a standard probability state space, with $\calC \subset \Proba(E)$ a given convex subset of probability distributions containing the unique entropy minimizing distribution $\mu_\ast = \argmin_{\calC} \Ent( \, . \mid \pi)$.\medskip
% 
% A subset of admissible distributions $\calA \subset \calC \subset \Proba(E)$ be given, with at least one element having finite relative entropy with respect to $\pi$.\medskip  

Let $h \geq \Ent(\mu_\ast \mid \pi)$ be given, and assume that for all bounded measurable real-valued function $f:E \to \R$, and all $\eps > 0$, there exists an admissible distribution $\eta_{f,h,\eps} \in \calA \subset \calC$ such that
\begin{equation}\label{eq:ass_main}
\abs{ \Ent( \eta_{f,h,\eps} \mid \pi) - h } \leq \eps,  \quad \& \quad \eta_{f,h,\eps}(f) \geq \mu_\ast(f)- \eps.
\end{equation}

Then one first has that:
$$
\mrm{WLC}_{h}(\mu_\ast \mid \pi) =  - \Ent(\mu_\ast \mid \pi).
$$
has:
\begin{equation}\label{eq:cond_main}
\forall \mu \in \Proba(E), \qquad \mrm{WLC}_{h}(\mu \mid \pi) \geq \mrm{WLC}_{h}(\mu_\ast \mid \pi) + \Ent(\mu \mid \mu_\ast).
 %\sup_{\eta \in \calA \bigcap \set{\Ent( \, . \,  \mid \pi) \leq e }}  \Ent(\eta \mid \mu).
\end{equation}
where in the above the log-cost is defined by~\eqref{eq:log-cost}.
\end{The}
The above result immediately implies our main finding that 

$$
 \mu_\ast = \mathop{\mrm{argmin}}_{\mu \in \Proba(E)} \mrm{WLC}_{h}(\mu \mid \pi) ,
$$
and when the condition~\eqref{eq:cond_main} is satisfied for any $h$, we also obtain that:
$$
 \mu_\ast = \mathop{\mrm{argmin}}_{\mu \in \Proba(E)} \limsup_{h \to + \infty} \mrm{WLC}_{h}(\mu \mid \pi) ,
$$
the quantity $\limsup_{h \to + \infty} \mrm{WLC}_{h}$ being a natural definition of the worst-case log-cost when no constraint is applied to admissible targets.

\begin{Rem}
 The main assumption~\eqref{eq:ass_main} in Theorem~\ref{th:main1} above have a nice geometric interpretation. $\mu_\ast$ is the unique distribution obtained as the intersection of two convex sets $\calC$ and $\set{\Ent( \mid \pi) \leq h_\ast}$, where $h_\ast=\Ent(\mu_\ast \mid \pi)$ is such that the latter two are 'tangent' with each other. The condition~\eqref{eq:ass_main} exactly asks that in \emph{any} open\footnote{for the locally convex weak topology on finite measures making evaluation of measurable bounded function continuous. This topology is sometimes called the weak $\tau$-topology as in~\cite{dembo2009large}} half-space of distributions that contains $\mu_\ast$, one can find an element of $\calA$ with prescribed entropy.
\end{Rem}

\begin{proof}[Proof of Theorem~\ref{th:main1}] First recall that
$
\Ent(\mu_\ast \mid \pi) \leq \inf_{\eta \in \calA}\Ent(\eta \mid \pi)
$

\underline{Step~$1$.} We claim that
$$
\sup_{\eta \in \calA \bigcap \set{\Ent( \, . \,  \mid \pi) \leq h }}  \Ent(\eta \mid \mu_\ast) = h - \Ent(\mu_\ast \mid \pi).
$$
Indeed, let $\eta \in \calA$ be a probability such that $\Ent(\eta \mid \pi) < +\infty $. By the Pythagorean inequality~\eqref{eq:Pyth} one has $\Ent(\eta \mid \mu_\ast) < +\infty $ and $\eta\p{\abs{\ln \pi / \mu_\ast}}< +\infty$. One can then consider the following decomposition:
\begin{align}
 \Ent(\eta \mid \mu_\ast) &= \Ent(\eta \mid \pi) + \eta(\ln \frac{\pi}{\mu_\ast}), \label{eq:deco1} %\\
%&= \Ent(\eta \mid \pi) - e_\ast + \eta(\ln\frac{\mu_\ast}{\mu}), \label{eq:deco}
 \end{align}
and~\eqref{eq:Pyth} becomes $\eta(\ln \frac{\pi}{\mu_\ast}) \leq - \inf_{\eta \in \calA}\Ent(\eta \mid \pi) \leq - \Ent(\mu_\ast \mid \pi)$. By~\eqref{eq:ass_main}, there exists a sequence in $\calA$ such that $\lim_n \Ent(\eta_n \mid \pi) = h$ and $\lim_n \eta_n(\ln \frac{\pi}{\mu_\ast}) \geq -  \Ent(\mu_\ast \mid \pi)$. Hence the claim.

\underline{Step~$2$.} We claim that if $\mu \neq \mu_\ast$, then
$$
\sup_{\eta \in \calA \bigcap \set{\Ent( \, . \,  \mid \pi) \leq h }}  \Ent(\eta \mid \mu) > h - \Ent(\mu_\ast \mid \pi).
$$

Indeed, let $\eta$ with $\Ent(\eta \mid \pi) < +\infty$, and $\mu \neq \mu_\ast$ be given, and assume without lost of generality that $\Ent(\eta \mid \mu) < +\infty$, which implies $\eta\p{\abs{\ln \frac{\mu}{\pi}}}< +\infty$. One has the decomposition
\begin{align}
 \Ent(\eta \mid \mu) &= \Ent(\eta \mid \pi) + \eta(\ln \frac{\pi}{\mu}) \label{eq:deco} %+ \eta(\ln\frac{\mu_\ast}{\mu}) \nonumber \\
%&\leq \Ent(\eta \mid \pi) - \Ent(\mu_\ast \mid \pi) + \eta(\ln\frac{\mu_\ast}{\mu}), \label{eq:deco}
 \end{align}
%in the last line we have used alog-cost the Pythagorean inequality~\eqref{eq:Pyth}. 
By~\eqref{eq:ass_main}, there exists a sequence such that $\Ent(\eta_n \mid \pi) = h$ and $\lim_n \eta_n(\ln \pi / \mu) \geq \mu_\ast(\ln \pi / \mu )=  -\Ent(\mu_\ast \mid \pi) + \Ent(\mu_\ast \mid \mu)$. Hence the claim.
\end{proof}

\subsection{A sufficient condition in the atomless case}

In this section, we will show that the main assumption in Theorem~\ref{th:main1} is satisfied in a quite generic context related to applications. This is Theorem~\ref{th:main2} below, which is the second main theorem of this paper.\medskip %Loosely speaking, we will assume that the convex set $\calC$ of target distribution lies in an affine sub-space whose complementary is sufficiently 'large'; that is to say the \textit{a priori} information on the target distributions is not too precise. \medskip

Let $T:E \to F$ be a measurable map onto a secondary (standard Borel) measurable space $F$. Let $\calC_T$ denotes any convex subset of $\Proba(F)$. Let us assume that the convex set containing admissible target distributions is defined as the pull-back by $T$ of $\calC_T$:
\begin{equation}\label{eq:CT} \calC = \set{ T \sharp \eta \in \calC_T \subset \Proba(F)}.
\end{equation}
This is a completely generic context; any convex set $\calC$ can be written in this way, for instance trivially is a measurable isomorphism. This corresponds to the practical situation in which the prior information on the possible targets is given by a condition on the push-forward by $T$.

\begin{The}\label{th:main2} Let $(E,\pi)$ be a standard probability space. Let $\calC$ be given by~\eqref{eq:CT} and contains the unique associated entropy minimizing distribution. Assume that:
\begin{enumerate}
\item The conditional distributions $\pi(\, \, .\mid T=t)$ are atomless $T\sharp(\d t)$-almost everywhere.
\item The set of admissible target distributions $\calA$ contains at least all distributions having an indicator density with respect to $\mu_\ast$.
% % reference distribution of the form
% $$
%      \int_{} \pi\p{\,\, . \mid T=t} \eta_T(\d t),
% $$
% where $\eta_T \in \calC_T$ is a reference distribution with $\Ent(\eta_T \mid T\sharp \pi)=h_0 < +\infty$.   % with distributions having an indicator density wrt $\pi$.
% % \item  $T \sharp \pi$ have no atom in $C$.
\end{enumerate}

Then for any $h \geq \Ent(\mu_\ast \mid \pi)=h_\ast$, any $f:E \to \R$ bounded measurable, and any $\eps >0$, $\calA$ contains a distribution $\eta_{f,h,\eps}$ such that~\eqref{eq:ass_main} is verified. In particular, the statement of Theorem~\ref{th:main1} holds true for any $h \geq h_\ast$.
\end{The}
% \begin{Rem} $\eta_T$ has typically an indicator density with respect to $\pi$, of the form $$ \eta_T \propto \one_{{T \in C \subset F}} \pi, $$ for some $C \subset F$.
%  
% \end{Rem}

\begin{proof} We first claim that: i) the conditional distribution of entropy minimizer $\mu_\ast$ is the same as the conditional distribution of $\pi$:
\[
\mu_\ast( \d x \mid T(x) = t) = \pi( \d x \mid T(x) = t), \qquad T \sharp \pi( \d t) - \text{a.e.},
\]

%of the form: given on $E\times F$ by
%\[
% \mu_\ast(\d x) \delta_{T(x)}(\d t) = \pi(   \mid T(x) = t) T \sharp \mu_\ast( \d t) 
%\]
and ii) the entropy minimizer on $\calC_T$ relative to $T \sharp \pi$ is $T \sharp \mu_\ast$ is. Indeed, chain rule of conditional entropy reads:
$$
\Ent( \mu \mid \pi) = \int_F \Ent( \mu( \,\, . \mid T=t) \mid \pi( \,\, . \mid T=t))  T \sharp \mu( \d t) + \Ent( T \sharp \mu \mid T \sharp \pi);
$$
for $\mu$ ranging in $\calC$, the first term of the right hand side uniquely ($T \sharp \pi$ almost everywhere) vanish for $\mu( \,\, . \mid T=t) = \pi( \,\, . \mid T=t)$, while second term is minimized only if $T \sharp \mu$ minimizes entropy on $\calC_T$ relative to $T\sharp \pi$.
\medskip

Next, let $f$ be an arbitrary bounded measurable and $h \geq  h_\ast$ be given. The atomless assumption ensures that there exists a set $A_{f,h} \subset E$ satisfying for $T \sharp \mu_\ast$ almost all $t$
$$
\pi(A_{f,h} \mid T=t) =\e^{h_\ast - h} \quad \text{and} \quad \pi(f \mid A_{f,h} \,\&\,  T=t) \geq \pi(f \mid T=t).
$$
and Lemma~\ref{lem:meas1} (based on the existence of an independent complement to $T$) ensures that $A_{f,h}$ can be chosen among Borel subsets.\medskip

One can then define a target distribution candidate to satisfy~\eqref{eq:ass_main}.
$$
\eta_{f,h}(\d x) \eqdef \mu_\ast(\d x \mid A_{f,h} ) %\int_{t \in \R^d} \frac{\one_{A_{f,h}}(x)}{ \pi(A_{f,h} \mid T=t)} \pi(\d x \mid T=t) T\sharp \mu_\ast ( \d t) 
= \e^{-h_\ast + h} \one_{A_{f,h}}(x) \mu_\ast(\d x).
$$

We can now conclude by showing that $\eta_{f,h}$ indeed satisfies~\eqref{eq:ass_main} for $\eps = 0$:\medskip

\underline{Step~$1$.}Since $\mu_\ast( \, \,  \mid T=t) = \pi( \, \, \mid T=t )$, one can remark that  by construction of $A_{f,h}$, one also have $\mu_\ast(f \mid A_{f,h} \,\&\,  T=t) \geq \mu_\ast(f \mid T=t)$, so that it holds $\eta_{f,h}(f) \geq \mu_\ast(f)$. \medskip

\underline{Step~$2$.} We remark that by construction $\pi(A_{f,h} \mid T=t)$ is independent of $t$ so that the push-forward distribution is unchanged $T \sharp \eta_{f,h} = T \sharp \mu_\ast$. Moreover, since $\eta_{f,h}( \, \, . \mid t=T)$ has an indicator density with respect to $\mu_\ast( \, \,  \mid T=t) = \pi( \, \, \mid T=t )$, a routine calculation yields that 
$$
\Ent(\eta_{f,h}( \,\, . \mid T=t ) \mid \pi( \,\, . \mid T=t )) =
- \ln \eta_{f,h}(A_{f,h} \mid T=t).
$$
Using again the chain rule of conditional entropy we obtain 
\begin{align*}
\Ent(\eta_{f,h} \mid \pi) %&= \Ent(T\sharp \eta \mid T \sharp \pi ) + \int \Ent(\eta(  \mid T=t ) \mid \pi(  \mid T=t )) \d T \sharp \eta \\
& = \Ent(T \sharp \mu_\ast \mid T \sharp \pi) + \int - \ln \pi(A_{f,h} \mid T= t) T \sharp \eta(\d t) \\
& = \Ent(\mu_\ast \mid \pi) + h - h_\ast = h.
\end{align*}
% Let $t \mapsto \pi(\, . \, \mid T=t)$ be a measurable representative defined for al $t \in \R^d$. $t \mapsto  \pi(f \mid T=t)$ is measurable and so is $x \mapsto \pi(f \mid T=T(x))$.
% 
% 
% since $\mu_\ast \ll \pi$ the set $\set{f \geq \mu_\ast(f)}$ has strictly positive measure with respect to $\pi$.
%  
\end{proof}

\subsection{A simple two atoms counter-example}
We next provide a very simple counter-example showing that the atomless asumption is critical in Theorem~\ref{th:main2}. For simplicity we consider $$\calA=\calC=\Proba(E),$$
so that entropy minimizing distribution is simply the reference $\pi$:
$$
\mu_\ast = \pi, \qquad h_\ast = 0.
$$
Similar counter-examples can be constructed in more complex cases. \medskip 

The problem arises when the discrete structure is not uniform: there are atoms with very different masses. In that discrete case, the (absolute) worst-case target distribution for the proposal $\pi$ is exactly a Dirac distribution on the atom with smallest probability. The worst-case optimal proposal is then the uniform discrete measure: proposing with $\pi$ is sub-optimal because of the smallest atom. However, for medium values of $h$, the discrete uniform distribution competes with the reference $\pi$; while for smaller values of $h$, targets concentrated on the smallest atom are not allowed and $\pi$ is optimal.\medskip

%It also shows that for non-uniform discrete spaces it might be more robust to construct proposals closer to the uniform discrete measure rather than the reference measure $\pi$.

\begin{Pro}
Let $E=\set{1,2}$ and $\pi \in \Proba(E)$ non-degenerate with $\pi(1)\geq\pi(2) > 0$. Let us denote $\mrm{WLC}_h(\mu \mid \pi) \eqdef \sup_{\eta: \, \Ent(\eta \mid \pi) \leq h} \Ent(\eta \mid \mu)$. Remark that $\Ent(\delta_i \mid \pi) = - \ln \pi(i)$. If $h \in [- \ln \pi(1), - \ln \pi(2)]$, let us denote by $\pi_h$ the unique distribution such that
$$
\Ent(\pi_h \mid \pi) = h.
$$
If $h \leq - \ln \pi(1)$ then
$$
\pi = \argmin_{\mu \in \Proba(E)} \mrm{WLC}_h(\mu \mid \pi), 
$$
If $
h \in [ - \ln \pi(1), -\ln \pi(2)]
$
, then
$$
\pi_h = \argmin_{\mu \in \Proba(E)} \mrm{WLC}_h(\mu \mid \pi), 
$$
If $h \geq - \ln \pi(2)$ then 
$$
\mrm{Unif}(\set{1,2})= \argmin_{\mu \in \Proba(E)} \mrm{WLC}_h(\mu \mid \pi)).
$$
\end{Pro}
\begin{proof} We start with the case $h \leq - \ln \pi(1) \leq - \ln \pi(2)$. We claim that the assumptions of Theorem~\ref{th:main1} are satisfied. Let $f:E \to \R$ be given, and let $i \in \set{1,2}$ be such that $f_i = \max f$. The map $\theta \mapsto \Ent(\theta \pi + (1-\theta) \delta_i \mid \pi)$ is continuous on $[0,1]$ taking values in $[0,-\ln \pi(i)]$ so that the value $h \leq - \ln \pi(i)$ is attained; and we set $\eta_{f,h} = \theta_h \pi + (1-\theta_h) \delta_i $ where $\theta_h$ is such that $\eta_{f,h} = \pi_h$. One also has $\eta_{h}(f) \geq \pi(f)$ by construction and the claim follows.\medskip
 
The case $h \geq - \ln \pi(2)$ is quite simple. The map $\eta \mapsto \Ent(\eta \mid \mu)$ is strictly convex with two admissible extrema: $ \eta = \delta_1$ and $\eta = \delta_2$ suprema with respective entropies $- \ln \mu(1)$ and $- \ln \mu(2)$. This implies $\mrm{WLC}_h(\mu \mid \pi) = - \ln \min(\mu(1),\mu(2)) \geq - \ln 1/2$, this last inequality being an equality if and ony if $\mu$ is the uniform distribution. This yields the result.\medskip

In the case $
h \in [ - \ln \pi(1), -\ln \pi(2)]
$, one also consider the continuous strictly convex map $\eta \mapsto \Ent(\eta \mid \mu)$ which has now two admissible extrema: i) the Dirac distribution $ \eta = \delta_1$ on the one hand, and ii) $\pi_h$. The associated admissible extrema defining $\mrm{WLC}_h$ are then respectively:  $- \ln \mu(1)$ and $h$. $h$ is the greatest so that $\mrm{WLC}_h(\mu \mid \pi) = \Ent(\pi_h \mid \mu)$. Hence the result.

\end{proof}

\section{Gibbs exponential families and applications} \label{sec:gibbs}

In this section, we consider a slightly more practical viewpoint. We assume that the main objective is to set up an importance sampling method aiming at numerically estimate averages with respect to a given target distribution $\eta_{\mrm{true}}$, defined up to a normalizing constant,
$$
\eta_{\mrm{true}}(\d x) \propto f(x) \pi(\d x).
$$
We first discuss Gibbs exponential families, which are the entropy minimizing proposals that arise when considering targets in a convex set defined using the push-forward by a vector-valued function $T$. Then, we will present an applicative scenario and discuss the relationship with the well-known \emph{cross-entropy} method.

\subsection{Gibbs exponential families}
Let us now consider a vector-valued bounded measurable function defined on the considered state-space:
$$
T: E \to \R^d.
$$
For instance, $T$ may be given by various statistics, or physical observables of special interest. We also assume that we know that averages of $T$ with respect to the target distribution belong to a given closed convex set $C$ of $\R^d$:
$$
\eta_{\mrm{true}}(T) \in C \subset \R^d.
$$
$C$ is assumed to have a non void intersection with the support of $T\sharp \pi$. This defines the admissible convex set of distributions $\calC = \set{\eta: \eta(T) \in C}$.
% In practice, $\eta_{\mrm{true}}$ is given using some structure. This structure defines the set of admissible targets $\calA$. In that context, assuming for instance following two conditions:
% \begin{enumerate}
%  \item $\calA$ is the set of indicator densities with respect to $\pi$ (a rare event). This implies that $\calA$ contains at least all the indicator densities with respect to $ \one_{T(x) \in C_{\eta^N}} \pi(\d x )
% $ which belongs to $\calC_{\eta^N}$.
%  \item The conditional distributions $\pi( \,\, . \mid T =t)$ are atomless distribution $T \sharp \pi$ almost everywhere,
% \end{enumerate}
If one assumes that the conditional distributions $\pi( \,\, . \mid T =t)$ are atomless distribution $T \sharp \pi$ almost everywhere, one can directly apply Theorem~\ref{th:main2} and obtain that the entropy minimizing distribution $\mu_\ast$ is the worst-case optimal proposal in the sense of the log-cost~\eqref{eq:log-cost}. 

$\mu_\ast$ is well-known and is given by 
$$
\mu_\ast = \frac{1}{Z_{\beta_\ast}}{\rm e}^{\langle \beta_\ast, T \rangle } \d \pi .
$$
It is the unique distribution in $\calC = \set{\eta: \eta(T) \in C}$ and in the Gibbs (a.k.a canonical) exponential family $\mu_\beta \propto \e^{\langle \beta, T \rangle } \d \pi$, $\beta \in \R^d$, which minimizes the entropy relative to $\pi$.\medskip %The optimal proposal in the exponential Gibbs family can be obtained by the routine expansion
%$$\Ent(\eta \mid \mu_\beta) = \ln Z(\beta)-\langle \beta, \eta(T) \rangle+\Ent(\eta \mid \pi),$$ followed by convex minimization with respect to $\beta$.  \medskip

%Under the following two conditions, which are arguably natural in practice, the main statement of Theorem~\ref{th:main1} holds true, and $\mu_\ast$ is indeed the worst-case optimal proposal:
%\begin{enumerate}
% \item The set of admissible (candidate) target distributions $\calA$ contains at least all indicator densities with respect to $\mu_\ast$ defined by subsets $A \subset \set{ T \in C}$:
% $$
% \eta(\d x) \propto \one_A(x) {\rm e}^{\langle \beta_\ast, T(x) \rangle } \pi( \d x).
% $$
% \item The conditional distributions $\pi(\,\, . \mid T=t)$ are, $T\sharp \pi( \d t)$-everywhere, atomless distributions.
%\end{enumerate}

It is also interesting to remark that it is not necessary to be able to sample exactly according to $\mu_\ast$ in order to implement the method. Indeed, $\mu_\ast$ is described by $T$ up to a normalization, so that a Sequential (or other) Monte Carlo routine (\textit{e.g.} \cite{del2006sequential}) can be used to sample according to the latter and estimate the associated normalization. The latter can be coupled with an iterative convex minimization routine calculating $\beta_\ast$ 

\subsection{An applicative context}

In some practical cases, the preliminary information on the target distribution $\eta_{\mrm{true}}$ may be given by preliminary importance sampling in the form of a weighted empirical probability distribution
$$ 
\eta^N = \sum_{n=1}^N W^n \delta_{X_n},
$$
defined by a sequence of random states $X_1, \ldots, X_N \in E^N$, on which the density $f(X_1), \ldots, f(X_N)$ have been previously evaluated.\medskip

Let us again consider a vector-valued bounded measurable function defined on the considered state-space:
$
T: E \to \R^d.
$
For instance, $T$ may be given by various statistics, or physical observables of special interest. It may happen then that the user is able to define a confidence convex set $C_{\eta^N} \subset \R^d$ -- which heavily depends on the precise construction of $\eta^N$ --  which asserts that with very high probability
\begin{equation}\label{eq:confidence}
  \P[\eta_{\mrm{true}} \in \calC_{\eta^N}] \simeq 1
\end{equation}
where one has defined the confidence set of target distributions
$$
\calC_{\eta^N} \eqdef \set{\eta \in \calP(E):  \eta(T) \in C_{\eta^N} \subset \R^d}.
$$
One can then choose as a worst-case optimal proposal the Gibbs distribution $\mu_\ast \propto {\rm e}^{\langle \beta_\ast, T \rangle } \d \pi$ minimizing entropy with the constraint that averages of $T$ belong to $C_{\eta^N}$. This optimal proposal can then be computed using various combination of Monte Carlo sampling and convex optimization routines. Although each ingredient (sampling and convex optimization) are well-known, such combinations and the proposed application are not standard up to our knowledge, and a detailed study is left for future work. \medskip

%apply Theorem~\ref{th:main2} in order to choose the worst-case optimal proposal. The latter is given by the Gibbs distribution $\mu_\ast \propto {\rm e}^{\langle \beta_\ast, T \rangle } \d \pi$ minimizing entropy on the convex set of admissible $\beta$. $\beta_\ast$ 

An interesting point consists in remarking that the proposed optimization of the proposal can be done iteratively, leading to a kind of adaptive importance sampling algorithm. This typically also happens in the so-called \emph{cross-entropy method} that we are to briefly discuss below. 

\begin{comment}
The general idea is that  and $\eta^N$ is constructed using the full sequence of   One then can then try to find the 'optimal' importance proposal $\mu_\ast$ to sample $\eta_{\mrm{true}}$ without resorting to new information or evaluation of $\eta_{\mrm{true}}$.\medskip

It may happen then that the user is able to define a confidence convex set $C_{\eta^N} \subset \R^d$ -- which heavily depends on the precise construction of $\eta^N$ -- centered at $t_N = \eta^N(T)$, and which asserts that with very high probability, the target distribution belongs to the following class of admissible targets.

Note that $\pi(T \in C_{\eta^N})$ must be strictly positive.
One can then apply Theorem~\ref{th:main2} in order to find the optimal proposal. The latter is given
\end{comment}

\subsection{Relation to the cross-entropy method}\label{sec:cross-ent}

Broadly speaking, the cross-entropy method is an adaptive importance sampling method performed iteratively (with main step hereafter indexed $k$).  On this topic, the interested reader can refer to the book~\cite{rubinstein2004cross}, or to~\cite{de2005tutorial} for a shorter introduction. See also yhe related adaptive importance sampling in~\cite{douc2007convergence}. \medskip

Let $\eta^N_{(k)}$ denotes a weighted empirical distribution approximating $\eta_{\mrm{true}}$ and constructed with a previously obtained samples, as discussed in the previous section. In cross-entropy methods, importance proposals are chosen in a parametric family $(\mu_\theta)_{\theta \in \Theta}$, and in the most basic versions of the algorithm, at step $k$, an i.i.d. sequence $X^{(k)}_1, \ldots,X^{(k)}_N$ is sampled with distribution $\mu_{\theta^N_{{(k-1)}}}$ for an iteratively chosen parameter $\theta^N_{{(k-1)}}$. The approximating weighted empirical distribution discussed in the previous section is then explicitly given by
$$
\eta_{(k)}^N \propto \sum_{n=1}^N \frac{f}{\d \mu_{\theta^N_{{(k-1)}}}}\p{X^{(k)}_n} \delta_{X^{(k)}_n}.
$$

Finally, the specific parameter $\theta^N_{(k)}$ at which importance sampling is performed at step $k+1$ is chosen by solving the minimization problem
\begin{equation}\label{eq:cross-entropy}
\theta_{(k)}^N = \argmin_\theta \Ent \p{ \eta^N_{(k)} \mid \mu^N_{(k),\theta} },
\end{equation}
where  in the above we denote the empirical version of a proposal by
$$
\mu^N_{(k),\theta} \propto \sum_{n=1}^N \frac{\d \mu_{\theta}}{\d \mu_{\theta^N_{{(k-1)}}}}\p{X^{(k)}_n} \delta_{X^{(k)}_n}.
$$
\medskip

In many cases, (with the obvious exception of mixtures), $\mu_\theta$ is chosen in an exponential family that can be set in a canonical Gibbs form:
$$
\mu_\theta \equiv \mu_\beta \propto {\rm e}^{\langle \beta, T \rangle } \d \pi. 
$$
At each step of the method, the minimizer is then precisely given by $\mu_\ast=\mu_{\beta_\ast}$ where $\beta_\ast$ is the unique vector such that
$$
\mu_{\beta_\ast}(T) = \eta^N(T) .
$$

The novelty of the present paper consist in showing that the distribution $\mu_{\beta_\ast}$, according to Theorem~\ref{th:main2}, is also the worst-case optimal proposal \emph{among all possible proposals}, as soon as one considers the following convex set of admissible targets: 
$$
\calC_{\eta^N} = \set{\eta \in \Proba(E): \, \eta(T) = \eta^N(T)}.
$$
In other words, we have thus shown that the minimization of relative entropy between an empirical target $\eta^N$ and proposals in an exponential family is \emph{equivalent} to the minimization of \emph{worst-case} relative entropy for targets $\eta$ with fixed average $\eta(T)=\eta^N(T)$; the minimization being obtained for proposals spanning \emph{all} distributions. This seems to be an original interpretation of the cross-entropy method.\medskip

Our result also suggests a possible way to improve or control cross-entropy methods. In practice, confidence intervals in the form of~\eqref{eq:confidence} could be used instead of the singleton $C = \set{\eta^N(T)} \subset \R^d$. This might improve the robustness of cross-entropy algorithms, especially in high dimension. 

%limit the possible choice of dimensionality of the statistics $T$ -- the higher the dimension, the worse the confidence interval--, and it may happen that a compromise between over-fitting and under-fitting can be looked for by minimizing $$\inf_{\mu,(C,T):\, \mu(T) \in C} \Ent(\mu \mid \pi)$$ for various pairs $(T,C)$. This is left for further investigation.

\section{The sample size required for importance sampling}\label{sec:sample_size}

In this section, $\eta \in \Proba(E)$ will denote a generic target probability distributions and $\mu \in \Proba(E)$ a generic proposal distribution.\medskip 

In the so-called \emph{importance sampling} method, the density $\d \eta / \d \mu$ is assumed to be computable (perhaps only up to a normalizing constant), and averages of the form $\eta(\ph)$ are estimated using the empirical distribution of $n$ i.i.d. variables $(X_i)_{i=1 \ldots n}$ distributed according to $\mu$. The estimator is given by:
\begin{equation}\label{eq:estim}
 \eta^N(\ph) \eqdef \frac1N \sum_{i=1}^N \ph(X_i) \, \frac{\d \eta}{\d \mu}(X_i) \xrightarrow[n \to +\infty]{\text{a.s.}} \eta(\ph).
\end{equation}
In the above, only the product $\ph \,  \frac{\d \eta}{\d \mu}$ has to be evaluated numerically; if the normalizing constant $Z$ is unknown, the test function $\ph$ must be defined as a product of this unknown normalization $Z$ with another computable function. From now on, we will denote
$$
Y \eqdef \frac{\d \eta}{\d \mu}(X) \geq 0, \qquad X \sim \mu.
$$
In the special case when the test function $\ph = Z$ is the normalizing constant, the relative variance of the estimation of $Z$ can be immediately computed:
\begin{equation}\label{eq:var}
\Var(\eta^N(\one)) = \frac{\Var(Y)}{N} =  \frac{\e^{\Ent_2(\eta \mid \mu)}-1}{N}
\end{equation}
where $\Ent_2(\eta \mid \mu) = \ln \eta \p{\frac{\d \eta}{\d \mu}}$ is the order~$2$ R\'enyi entropy (see below for a definition). Using~\eqref{eq:var} and various standard concentration inequalities, one can then obtain an upper bound on (some appropriate notion of) the \emph{required} sample size $N^\ast$. This upper bound is usually comparable to the variance (using Chebyshev inequality)
$
N^\ast \leq c \Var(Y),
$
or to the square of an upper bound on the support of $Y$
$
N^\ast \leq c \ell^2
$
if  $Y \leq \ell$,
or a combination of both; in the above $c$ is a numerical constant.

\subsection{The Chatterjee-Diaconis bounds}

We shall use here an alternative approach motivated by~\cite{ChaDia}. In the latter reference, some theorems (Theorem~$1.1$, $1.2$ and $1.3$) are proved showing that the sample size $N^\ast=N^\ast_{\delta,p_\alpha}(\eta \mid \mu)$ required to obtain an importance sampling estimation at a given precision $\delta$ and with a given probability $p_\alpha$ is, quite generically, and at logarithmic scales, given by the relative entropy of the target $\eta$ with respect to the proposal $\mu$:
\begin{equation}\label{eq:sample_size}
\ln N^\ast \simeq \Ent(\eta \mid \mu) = \E( Y \ln Y).
\end{equation}

The failure (or equivalently success) probability of importance sampling at a given precision  threshold $\delta > 0 $ is rigorously defined as follows. The deviation probability of importance sampling is defined using $\L^2(\nu)$-normalized test functions, and is given by:
\begin{equation}\label{eq:failgen}
p_{\mrm{dev},\delta}(N) \eqdef \sup_{\ph: \, \eta(\ph^2)=1} \P\p{ \abs{\eta^N(\ph) - \eta(\ph)} \geq \delta  },
\end{equation}
where one considers estimators of the form~\eqref{eq:estim}. A sample size denoted $N^\ast = N_{p_\alpha,\delta}^\ast(\eta \mid \mu)$ and satisfying
$$
p_{\mrm{dev},\delta}(N^\ast) = p_\alpha.
$$
% 
% $$
% \set{n \in \N^\ast \mid p_{\mrm{dev},\delta}(n) \in [\eps, 1-\eps]}.
% $$
is called a critical sample size \emph{required} for importance sampling. It is a sample size achieving $\delta$-accuracy with success probability exactly $1-p_\alpha$. \medskip %, the upper bound is, symmetrically, the maximal sample size achieving $\delta$-accuracy with prescribed failure probability.

For simplicity, we now state a slightly weaker version of Theorem~$1.1$~\cite{ChaDia} (based on lower order moment conditions)
for the special case where the failure probability is defined with the estimator of the normalizing constant rather than~\eqref{eq:failgen}:
$$
p_{\mrm{dev},\delta}(N) \eqdef \P \p{ \p{\eta^{N}(\one) - 1} \geq \delta }.
$$
This case enables to present the result as a general result on sum i.i.d. real valued random variables. The case~\eqref{eq:failgen} can be treated in the same way (since the lower bound in Theorem~$1.1$~\cite{ChaDia} is obtained for constant test functions $\ph= \one$). Such results and their proofs can be found in Section~\ref{app:ChaDia}. 

\begin{The}[Corollary of Theorem~$1.1$~\cite{ChaDia}] Let $Y_i \geq 0$, $i=1 \ldots N$ i.i.d. random variables with unit mean $\E Y = 1$. Denote by $N^\ast$ a sample size verifying
\[
\P \p{ \abs{ \frac1{N^\ast} \sum_{i=1}^{N^\ast} Y_i - 1} \geq \delta } = p_\alpha,
\]
for some $\delta,p_\alpha \in (0,1)$. Then the following estimate holds
\begin{equation}\label{eq:main_bound}
 \abs{ \ln N^\ast  - \E(Y \ln Y) } \leq \inf_{\theta \in [0,1]} \b{ \frac{1}{\theta} \ln \p{\E(Y^{1+\theta}) \E(Y^{1-\theta)) } } + \frac{\ln c(\delta,p_\alpha)}{\theta} },
\end{equation}
where $c_{\delta,p_\alpha}$ is a numerical constant depending only on $p_\alpha$ and $\delta$.
\end{The}
The estimate~\eqref{eq:main_bound} result can be interpreted a concentration/anti-concentration inequality with lower order moments condition. The proof, discussed in Section~\ref{app:ChaDia}, is based on: i) an upper bound on the deviation probability $p_{\mrm{dev},\delta}(N)$ that decreases with the sample size according to a (slow) power-law $N^{-\theta/4}$, ii) an upper bound on the success probability $1-p_{\mrm{dev},\delta}(N)$ that increases with the sample size according to a (slow) power-law $N^{\theta/2}$ . \medskip

In the context of importance sampling, one has $\frac{1}{N^\ast} \sum_{i=1}^{N^\ast} Y_i = \eta^N(\one)$, $\E(Y \ln Y)=\Ent(\eta \mid \mu)$, and the logarithmic moments 
$$
\frac{1}{\theta} \ln \p{\E(Y^{1+\theta}} = \frac{1}{\theta} \ln \int \p{\frac{\d \eta }{\d \mu}}^\theta \d \eta \eqdef \Ent_{1+\theta}(\eta \mid \mu), \qquad \theta \in [-1,+\infty],
$$
are the order $1+\theta$ R\'enyi entropy of $\eta$ relative to $\mu$. The latter is increasing with $\theta$, and it is continuous when finite. $\Ent_{1} = \Ent$ is the usual relative entropy. We refer to~\cite{van2014renyi} for a review of properties of those entropies. \medskip

In our context, the estimate in~\eqref{eq:main_bound} is meaningful only in cases where the relative entropy $\E(Y \ln Y)$ is large as compared to the difference of R\'enyi entropies $\Ent_{1+\theta}(\eta \mid \mu)-\Ent_{1-\theta}(\eta \mid \mu)$. It is argued in ~\cite{ChaDia} that this situation happens quite generically, see the discussion in Section~\ref{app:ChaDia}. Some care however is required in practice because the constants in the right hand side of~\eqref{eq:main_bound} can be unsatisfactory; they are however certainly not sharp, and it may turn out that improved estimates can be obtained. \medskip

A conservative user of importance sampling might prefer to minimize the usual variance $\Var(Y)=\eta( \frac{\d \eta}{\d \mu}) -1$,  rather that $\Ent(\eta \mid \mu)$. However, minimizing upper bounds might lead to inefficient results depending on context. The estimate~\eqref{eq:main_bound} shows that $\Ent(\eta \mid \mu)$ is, loosely speaking, a compromise between an upper bound and a lower bound. This makes relative entropy an interesting practical criteria. \medskip

Moreover, denoting the right hand side in~\eqref{eq:main_bound} by 
$$ R(Y) \eqdef \inf_{\theta \in [0,1]} \b{ \frac{1}{\theta} \ln \p{\E(Y^{1+\theta}) \E(Y^{1-\theta}) } + \frac{\ln c}{\theta} },$$
we provide in the next section an important class of examples for which, in a quite generic asymptotics, the following holds:
% \begin{equation}\label{eq:exa_as1}
% \ln \Var(Y) \gg \E(Y \ln Y) \gg R(Y)
% \end{equation}
% as well as
\begin{equation}\label{eq:exa_as}
\ln \Var(Y) - \E(Y \ln Y) \gg  R(Y).
\end{equation}
This implies, according to~\eqref{eq:main_bound}, that the sample size estimate with $\e^{\E(Y \ln Y)}$ is arbitrarily more accurate than $\Var(Y)$, in the sense that:
\[
 \abs{\ln N^\ast - \ln \Var(Y) }  \gg \abs{\ln N^\ast -  \E(Y \ln Y)},
\]
that is, in other words, variance is unwarrantedly too large.

\subsection{A motivating example}\label{sec:example}
Consider the class distribution given by
$$
\begin{cases}
& \P(Y=0) = 1- p_1 - p_2, \\
& \P(Y = l_1) = p_1, \\
& \P(Y = l_2) = p_2,
\end{cases}
$$
with the condition
$$
\alpha \eqdef p_1 l_1 = 1 -p_2l_2,
$$
which is equivalent to $\E Y = 1$. We assume that 
$$
r \eqdef \frac{l_2}{l_1} \to 0, \qquad l_1 \to +\infty,
$$
together with
$$
\alpha \to 0,
$$
which means that a large value $l_1$ has a small contribution $\alpha$ in the average $\E(Y)$.

Straightforward calculations yields:
$$
\ln \E(Y^2)= \ln(\alpha l_1 + (1-\alpha) l_2) = \ln l_1 + \ln(\alpha + (1-\alpha) r) ,
$$
and we now assume that
$$
\alpha \gg r^{1-\alpha} \geq r
$$
which implies that the large value $l_1$ has a dominating contribution in the variance (contrary to the mean). \medskip

One can now compute
$$
\E(Y \ln Y) = \alpha \ln l_1 + (1-\alpha) \ln l_2 = \ln l_1 + (1-\alpha) \ln r,
$$
so that the difference with the log-variance diverge:
$$
\ln \E(Y^2) - \E(Y \ln Y) \sim  \ln(\alpha) - (1-\alpha) \ln r = \ln \frac{\alpha}{r^{1-\alpha}} \to + \infty.
$$

On the other hand
$$
\frac{1}{\theta} \ln \E(Y^{1+\theta}) = \frac1\theta \ln \p{\alpha l_1^\theta + (1-\alpha)l_2^\theta }  = \ln l_1 + \frac1\theta \ln \p{\alpha + (1-\alpha) r^\theta },
$$
so that
$$
\frac{1}{\theta} \ln \p{\E(Y^{1+\theta}) \E(Y^{1-\theta}) } = \frac1\theta \ln \p{1+ \alpha(1-\alpha)(r^\theta+r^{-\theta}-2)}.
$$

Now, taking
$$
\theta = \frac{-\ln \alpha}{- \ln r} \to 0
$$
implies that $ \alpha r^{-\theta}$ remains bounded so that
$$
\inf_{\theta \in [0,1]} \b{ \frac{1}{\theta} \ln \p{\E(Y^{1+\theta}) \E(Y^{1-\theta}) }  + \frac{\ln c}{\theta} } \leq \frac{- \ln r}{- \ln \theta} \ln(2c),
$$
and since
$$
\frac{- \ln r}{- \ln \theta} \ll -\ln \frac{r^{1-\alpha}}{\alpha},
$$
which implies that ~\eqref{eq:exa_as} holds true. \medskip

\appendix
\section{More on the Chatterjee-Diaconis bounds}\label{app:ChaDia}
In this section, $\eta \in \Proba(E)$ will denote generic target probability distributions, and $\mu \in \Proba(E)$ a generic proposal. The importance sampling estimator is given by for any test function $\ph$.
\begin{equation*}
 \eta^N(\ph) \eqdef \frac1N \sum_{i=1}^N \ph(X_i) \, \frac{\d \eta}{\d \mu}(X_i) \xrightarrow[n \to +\infty]{\text{a.s.}} \eta(\ph).
\end{equation*}
The failure probability of importance sampling is defined using a supremum over $\L^2(\nu)$-normalized test functions, and is given, for a precision threshold $\delta$ by:
\begin{equation} \label{eq:pfail1}
p_{\mrm{dev},\delta}(N) \eqdef \sup_{\ph: \, \eta(\ph^2)=1} \P\p{ \abs{\eta^N(\ph) - \eta(\ph)} \geq \delta  }.
\end{equation}
One can alternatively consider the failure probability of estimating a normalizing constant
\begin{equation} \label{eq:pfail2}
p_{\mrm{dev},\delta}(N) \eqdef \P\p{ \abs{\eta^N(\one) - \eta(\one)} \geq \delta  }.
\end{equation}
If $p_\alpha$ denotes a failure probability, a critical sample size $N^\ast = N^\ast_{p_\alpha,\delta}(\eta \mid \mu) $ \emph{required} for importance sampling is then defined by 
\begin{equation}\label{eq:Ncrit}
p_{\mrm{dev},\delta}(N^\ast) = p_\alpha .
%N(\eta \mid \mu) = N_{p_\alpha,\delta}(\eta \mid \mu) \eqdef \argmin \set{n: \, p_{\mrm{dev},\delta}(N) \leq p_\alpha}.
\end{equation}
% 
% $$
% \set{n \in \N^\ast \mid p_{\mrm{dev},\delta}(n) \in [\eps, 1-\eps]}.
% $$
which is precisely the sample size achieving $\delta$-accuracy with success probability $1-p_\alpha$. \medskip

The main theorems of~\cite{ChaDia} and the subsequent examples are not so easy to state and interpret. We propose in this section a slightly weaker reformulation of the main result Theorem~$1.1$ of~\cite{ChaDia} using R\'enyi entropies. \medskip

R\'enyi entropies are the power-law generalization of relative entropy and are defined by 
$$
\Ent_{\alpha}(\eta \mid \mu) \eqdef \frac{1}{\alpha} \ln \int \p{\frac{\d \nu }{\d \mu}}^\alpha \d \nu, \qquad \alpha \in [0,+\infty] 
$$
if $\nu \ll \mu$, while $\Ent_{\alpha}(\eta \mid \mu)=+\infty$ otherwise. $\alpha \mapsto \Ent_{\alpha}(\nu \mid \mu)$ is increasing and continuous when finite. $\Ent_{1} = \Ent$ is the usual relative entropy. We refer to~\cite{van2014renyi} for a review of properties of those entropies. \medskip

A corollary of the main result (Theorem~$1.1$) in~\cite{ChaDia} is then the following:
\begin{Cor}[Theorem~$1.1$,~\cite{ChaDia}]\label{cor:ChaDia} Let $\eta,\mu \in \Proba(E)$ be two given probability distribution, and let $N^\ast$ be a sample size verifying~\eqref{eq:pfail1}-\eqref{eq:Ncrit} or \eqref{eq:pfail1}-\eqref{eq:Ncrit} with $\delta,p_\alpha \in (0,1)$. There is a numerical constant $c(\delta,p_\alpha) \leq \max\p{ \p{\frac{3}{p_\alpha \delta}}^{4},\p{\frac2{(1-p_\alpha)(1-\delta)}}^2}$ such that 

$$
\abs{\ln N^\ast - \Ent(\eta \mid \mu)} \leq \inf_{\theta \in [0,1]} \b{ 2 \p{\Ent_{1+\theta}(\eta \mid \mu) - \Ent_{1-\theta}(\eta \mid \mu) } + \frac{\ln c}{\theta} }
$$

\end{Cor}

%Let $\delta$ and $p_\alpha$ be fixed let us define 
%$ c(\delta,p_\alpha) =  \max\p{c_-(\delta,p_\alpha),c_+(\delta,p_\alpha) }$. 
Corollary~\ref{cor:ChaDia} implies that
$$
\ln N_{p_\alpha,\delta}(\eta \mid \mu) \mathop{\sim} \Ent(\eta \mid \mu)
$$ 
for asymptotics such that
\begin{equation}\label{eq:asympt}
\Ent(\eta \mid \mu) \mathop{\gg} \inf_{\theta \in (0,1)} \, \,  \Ent_{1+\theta}(\eta \mid \mu) - \Ent_{1-\theta}(\eta \mid \mu) + \frac{1}{\theta} \ln c(\delta,p_\alpha),
\end{equation}
where we stress that $c(\delta,p_\alpha)$ is numerical (independent of $\eta,\mu$). \medskip

In Section~$3$ of~\cite{ChaDia}, S.~Chatterjee and P.~Diaconis argued that many (high dimensional) toy models inspired by statistical mechanics indeed satisfy behaviors similar to~\eqref{eq:asympt} when $\Ent(\eta \mid \mu)$ is large. More specifically, one can first remark that R\'enyi entropies satisfy (if finite for $\abs{\theta}$ small):
$$
\Ent_{1+\theta}(\eta \mid \mu) = \Ent(\eta \mid \mu) + \frac{\theta}{2} \var_\eta( \ln \frac{\eta}{\mu} ) + \bigO(\theta^3).
$$
The authors then showed that in various cases of interest it holds
$
\var_\eta( \ln \frac{\eta}{\mu} ) = \bigO \p{ \Ent(\eta \mid \mu) = \eta( \ln \frac{\eta}{\mu}) } ,
$
which in other words means tha the mean and variance of the log-likelihood $\ln \frac{\eta}{\mu}$ are of the same order. In examples of Section~$3$ of~\cite{ChaDia}, these two quantities scale like the dimension of the considered system. Choosing $\theta$ of order
$$
\theta = \bigO( \Ent(\eta \mid \mu)^{-1/2})
$$
then yields the asymptotic condition~\eqref{eq:asympt}. We provided in Section~\ref{sec:example} a simpler, practically generic, example for which this condition is also satisfied.
\begin{proof}[Proof of Corollary~\ref{cor:ChaDia}] Let us recall the notation of Theorem~$1.1$ in~\cite{ChaDia}. The importance sampling estimator $\eta^N(\ph)$ is denoted $I^n(f)$. The authors also use the notation $L=\Ent(\eta\mid\mu)$ and $t = \pm(\ln N - L)$.We will also denote $\Ent_{1+\theta}(\nu \mid \mu) = L_{1+\theta}$.\medskip

Denote $t=\ln N - L \geq 0$. Using Markov inequality, the first inequality in Theorem~$1.1$ can be rewritten:
\[
p_{\mrm{dev},\delta}(N) \leq \frac1\delta\e^{-t/4}+\frac2\delta\sqrt{\eta(\one_{\ln \rho > D + t/2})}.
\]
A routine upper bound of Chernoff-type yields, for any $\theta \geq 0$,
\[
 \eta(\one_{\ln \rho > L + t/2}) \leq \e^{-\theta L - \theta t /2 } \underbrace{\eta(\rho^\theta)}_{\exp{\theta \Ent_{1+\theta}(\mu \mid \eta)}} = \e^{ \theta(L_{1+\theta} - L) - \theta t /2 } ,
\]
so that for $\theta < 1$:
\[
 p_{\mrm{dev},\delta}(N) \leq \frac3\delta\e^{-\theta t/4 + \theta(L_{1+\theta} - L)/2}.
\]
Denoting $ c =\p{\frac{3}{p_\alpha\delta}}^4$ for $p_{\mrm{dev},\delta}(N^\ast)=p_\alpha$, it yields
$$
\ln N - L = t \leq 2\p{L_{1+\theta} - L} + \frac{\ln c}{\theta},
$$ 
 the claimed upper bound on $\ln N^\ast - L$ follows. \medskip

Similarly, let us now denote $t=-\ln N + L \geq 0$. The second inequality in Theorem~$1.1$ in~\cite{ChaDia} can be rewritten:
\begin{align*}
1 - p_{\mrm{dev},\delta}(N) &\leq P(\eta^N(\one)-1 \geq - \delta) \\
& \leq \frac1{1-\delta}\e^{-t/2}+\frac{1}{1-\delta} \eta(\one_{\ln \rho \leq D - t/2}). \\
%& \leq \frac{2}{1-\delta} \e^{ - \theta t /2 + \theta (D-\Ent_{1-\theta}(\mu \mid \eta))}
\end{align*}
A similar routine upper bound of Chernoff-type yields
\[
 \eta(\one_{\ln \rho \leq L - t/2}) \leq \e^{\theta L - \theta t /2 } \underbrace{\eta(\rho^{-\theta})}_{\exp{\theta \Ent_{1-\theta}(\mu \mid \eta)}} = \e^{\theta (L-L_{1-\theta}) - \theta t /2 },
\]
so that 
\begin{align*}
1 - p_{\mrm{dev},\delta}(N) \leq \frac{2}{1-\delta} \e^{ - \theta t /2 + \theta (L-L_{1-\theta})}.
\end{align*}
Denoting $ c =\p{\frac{2}{(1-p_\alpha)(1-\delta)}}^2$ for $p_{\mrm{dev},\delta}(N^\ast)=p_\alpha$, it yields
$$
L-\ln N  = t \leq 2\p{L-L_{1-\theta} - L} + \frac{\ln c}{\theta},
$$ 
 the claimed lower bound on $\ln N^\ast - L$ follows. \medskip
\end{proof}

\section{A measurable selection lemma}
We first start by a general result of measure theory (see Theorem~$10.8.3$in Bogachev Vol.~$2$ \cite{bogachev2006measure2}), which states the existence of an independent ''complement'' of atomless conditional measures.
\begin{The}[Independent Complement]\label{the:comp} Let $T: E \to F$ denotes a measurable map between two standard Borel spaces. Let $\pi \in \Proba(E)$ denotes a probability such that for $T \sharp \pi$-almost all $t \in F$,  $\pi( \,\, .\mid T=t)$ is atomless. Let $\lambda$ denotes the usual Lebesgue measure on the interval $[0,1]$. Then there exists a measurable function $S: E \to [0,1]$ such that, i) $(T,S):E \to F \times [0,1]$ is one-to-one and, ii) 
$$
(T,S) \sharp \pi = T\sharp \pi \otimes \lambda;
$$
that is, if $X$ has distribution $\pi$, then $T(X)$ and $S(X)$ are independent with distribution $\p{T\sharp} \pi$ and $\lambda$ respectively. %In particular, there exists a representative $\pi( \,\, .\mid T=t)$ such that $S \sharp \pi( \,\, .\mid T=t)$ is the Lebesgue distribution for all $t \in F$.
\end{The}
We next use the above theorem to prove the existence of distributions in a given half-space with prescribed conditional entropy.

\begin{Lem}\label{lem:meas1} Let $T: E \to F$ denotes a measurable map between two standard Borel spaces. Let $\pi \in \Proba(E)$ denotes a probability such that for $T \sharp \pi$-almost all $t$,  $\pi( \,\, .\mid T=t)$ is atomless. For each $\eps \in (0,1)$ and all bounded measurable function $F:E \to \R$, there exists a measurable set $A_{\eps,F} \subset E$  verifying for $T \sharp \pi$-almost all $t$
$$
\pi(A_{f,h}\mid T = t) = \eps
$$
as well as
$$
\pi(F \mid A_{f,h},\, T=t) \geq \pi(F \mid T=t).
$$
\end{Lem}
\begin{proof} The existence a measurable complement (Theorem~\ref{the:comp}) implies \textit{a fortiori}, the existence of a measurable function $S:E \to [0,1]$ such that, for $T \sharp \pi$-almost all $t$, $S \sharp \pi( \,\, .\mid T=t)$ is the Lebesgue distribution. \medskip %, where $F_S$ is of $T \sharp \pi$-measure $1$. One can then conventionally set $\pi( \,\, .\mid T=t) = \pi$ for $t \in F_S^c$ to get a representative staisfying the same property.  so that conditional distributions to the Lebesgue measure: 
% $$
% \pi(s \geq z \mid T = t) = z;
% $$
% the latter being in particular obviously atomless.
Let us consider the pair $(y_\eps(t),z_\eps(t)) \in \R \times [0,1]$ defined by 
$$
\pi_t\p{\set{ F > y_\eps(t)} \cup \set{F = y_\eps(t) \,\&\, S \geq z_\eps(t) }} = \eps;
$$
The above equality being required for $t$ in a measurable set of $T \sharp \pi$-measure $1$ on which $S \sharp \pi( \,\, .\mid T=t)$ is the Lebesgue distribution. We conventionally set $(y_\eps(t),z_\eps(t))=(0,0)$ elsewhere. By Lemma~\ref{lem:meas2}, $t\mapsto (y_\eps(t),z_\eps(t))$ is a uniquely defined measurable map.\medskip

We can then consider the measurable set
\[
  A_{f,h} \eqdef  \set{x\in E: \, F(x) > y_\eps(T(x))} \cup \set{x\in E : \, F(x) = y_\eps(T(x)) \,\&\, S(x) \geq z_\eps(T(x)) } .
\]
% where 
% 
%  $$
%      \pi(f < y_\eps(t) \mid T=t) \leq \eps < \pi(f \leq y_\eps(t) \mid T=t)
%  $$
%  and $z_\eps(t)$ by
%  $$
%      \pi(f = y_\eps(t) \,\&\, s \leq z_\eps(t) \mid T=t) = \eps,
%  $$
%  where we have used the atomless assumption.\medskip
 
Let us now remark that for any subset $B \subset E$ 
the function $F$ takes greater values on the set
$$
A \eqdef \set{F > y} \cup \p{ \set{F=y} \cap B}
$$
than on its complementary $A^c$. Now for any probability $\eta$ on $E$
\[ 
\int F \d \eta = \eta(A) \d \eta(F\mid A) + (1-\eta(A)) \eta(F \mid A^c),
\]
which yields $\eta( F \mid A) \geq \eta(F)$. One can apply this  argument in our case to $\eta = \pi( \,\, . \mid T=t)$ and $A=A_{\eps,F}$ in order to obtain $\pi(F \mid A_{f,h}\, \& \, T=t) \geq \pi(F \mid  T=t)$ for $T\sharp \pi$-almost all $t$.
\end{proof}

Finally, we give an elementary lemma which proves the measurability of the set $A_{f,h}$ constructed in the proof of Lemma~\ref{lem:meas1} above.

\begin{Lem}\label{lem:meas2} Let $E$ and $F$ denotes two measurable spaces, and let $t\in F \mapsto \pi_t \in \Proba(E)$ denotes a measurable probability kernel ($t \mapsto \pi_t(A)$ is measurable for any measurable subset $A \subset E$). Assume given two measurable bounded functions $F: E \to \R$ and $G \to [0,1]$ such that the push-forward probability $G \sharp \pi_t$ is the Lebesgue measure for all $t \in F$. For each $t \in F$ and each $\eps \in (0,1)$, there exists a unique pair $(y_\eps(t),z_\eps(t)) \in \R \times [0,1]$ such that
$$
\pi_t\p{\set{ F > y_\eps(t)} \cup \set{F = y_\eps(t) \,\&\, S \geq z_\eps(t) }} = \eps;
$$
moreover, the map $t \mapsto (y_\eps(t),z_\eps(t))$ is measurable.
\end{Lem}
\begin{proof}
 Let us denote
 $$
 p_t(y,z) \eqdef \pi_t\p{\set{ F > y} \cup \set{F = y \,\&\, S \geq z }}.
 $$
 By a routine argument of measure theory, $(t,z,y) \mapsto p_t(x,y)$ is a measurable map as a bounded pointwise limit of simple functions in tensorial form. The map $(y,z) \mapsto p_t(y,z)$ is decreasing for the lexicographic order, and $z \mapsto p_t(y,z)$ is continuous and strictly decreasing since, by assumption, $S\sharp \pi_t$ is atomless. $y_\eps(t)$ can thus be defined as the unique $y$ such that $ \lim_{y+} p_t( \, . \, ,1) > \eps \leq p_t(y,1) $, and $z_\eps(t)$ the unique $z$ such that $p_t(y_\eps(t),z)=\eps$. This shows existence and uniqueness. \medskip
 
 In order to prove the measurability of $y_\eps,z_\eps$, it suffices to show that they are the monotone limit of measurable simple functions. \medskip
 
For each $k \in N$, let us denote the unique pair $(y_\eps^k(t),z_\eps^k(t))$ in $2^{-k} \Z \times 2^{-k} [0,\ldots,2^{k}]$ defined by
$$
p_t(y_\eps^k(t),1) \leq \eps <   p_t(y_\eps^k(t)-2^{-k},1),
%\pi(f < y_\eps^k(t) \mid T=t) \leq \eps < \pi(f \leq y_\eps^k(t)+\delta \mid T=t)
$$
as well as
$$
p_t(y_\eps(t),z_\eps^k) \leq \eps <   p_t(y_\eps(t),z_\eps^k(t)-2^{-k}).
$$
By construction, $k \mapsto y^k_\eps(t) $ and $k \mapsto z_\eps^k(t)$ are decreasing maps, and by $\sigma$-additivity, their respective infima are given by $y_\eps(t)$ and $z_\eps(t)$. \medskip

We now claim that $t \mapsto y^k_\eps(t)$ is measurable. Since the latter takes its values in a countable space, it suffices to show that the set $A_{y_0} \eqdef \set{t: \, y^k_\eps(t) = y_0}$ is measurable for each $y_0 \in \R$. But by definition $$A_{y_0} = \set{t: \, p_t(y_0,1) \leq \eps <   p_t(y_0-2^{-k},1) },$$ 
which is indeed measurable since $t \mapsto p_t(y,z)$ is measurable. As a consequence, $t \mapsto y^\eps_t$ is measurable as a decreasing limit of simple measurable functions. \medskip

We finally claim that $t \mapsto z^k_\eps(t)$ is also measurable. The argument works as in the paragraph above, except that we now use the measurability of the map $t \mapsto p_t(y_\eps(t),z)$.

%  as well as  
%  \begin{align*}
%      &\pi(f \in [y_\eps^k(t),y_\eps^k(t) + \delta] \,\&\, s < z_\eps^k(t) \mid T=t)  \\
%      & <  \eps + \pi(f \in [y_\eps^k(t),y_\eps^k(t) + \delta] \, \& \, s \leq z_\eps(k,t) + \delta \mid T=t).
%  \end{align*}
%  By construction, $y_\eps^k(t)$ and $z_\eps^k(t)$ are monotone functions of $\delta$ such that:
% \[
% \set{f < y_\eps(t)} =  \bigcup_{\delta=2^{-k}} \set{f < y_\eps^k(t)}
% \]
% as well as
% \[
%   \set{f = y_\eps(t) \,\&\, s \leq z_\eps(t) } = \bigcap_{\delta=2^{-k}} \set{f \in [y_\eps^k(t),y_\eps^k(t)+\delta] \,\&\, s \leq z_\eps^k(t)+\delta }.
%  \]
%  It is thus sufficient in order to conclude taht $A_{f,h}$ is measurable to show that $\set{f \leq y_\eps^k(t)}$ and $\set{s \leq z_\eps^k(t)}$ are measurable. Since $s$ and $f$ are measurable, it is in fact sufficient to check that $(y_\eps^k(t),z_\eps ^k(t))$ are measurable functions of $t$. The latter a is a consequence of $t \mapsto \pi ( \mid T=t)$ being measurable and $(y_\eps^k(t),z_\eps ^k(t))$ taking values in a countable space. \medskip

\end{proof}

\section*{Acknowledgement} This work has been partially supported by ANR SINEQ, ANR-21-CE40-0006.

\bibliographystyle{plain}
\bibliography{mathias}
\end{document}